\documentclass[reqno,11pt]{amsart}
\makeatletter
\def\section{\@startsection{section}{1}%
	\z@{.7\linespacing\@plus\linespacing}{.5\linespacing}%
	{\bfseries
		\centering
}}
\def\@secnumfont{\bfseries}
\makeatother

\usepackage{amsmath}
\usepackage{amsfonts}
\usepackage{amssymb}
\usepackage{graphicx}
\usepackage{xcolor}
\usepackage{soul}
\usepackage{ulem}
\usepackage{lineno}
\newtheorem{theorem}{Theorem}[section]

\newtheorem{definition}[theorem]{Definition}

\newtheorem{proposition}[theorem]{Proposition}
\newtheorem{remark}[theorem]{Remark}

\numberwithin{equation}{section}
\usepackage{lmodern}
\usepackage[colorlinks=true, allcolors=blue]{hyperref}
\colorlet{blu1}{blue!70!black}
\colorlet{blu2}{blue!50!black}
\colorlet{blu3}{blue!70!red}
\colorlet{blu4}{blue!60!green}
\colorlet{red1}{red!80}
\colorlet{red2}{red!50!black}
\colorlet{red3}{red!70!yellow}
\colorlet{red4}{red!50!yellow}
\colorlet{yel1}{yellow!50!black}
\colorlet{yel3}{yellow!20!blue}
\colorlet{gre1}{green!60!blue}
\colorlet{gre2}{green!60!black}
\colorlet{gre3}{green!40!black}
\catcode `\@=12 \textwidth16cm \textheight23cm \hoffset-1cm   \voffset-2.0cm
\begin{document}
\begin{center}
{\bf\Large Weighted Catalan convolution and $(q,2)-$Fock space} \vspace{1.5cm}

{\bf\large Yungang Lu}\vspace{0.6cm}
	
{Department of Mathematics, University of Bari ``Aldo Moro''}\vspace{0.3cm}
	
{Via E. Orabuna, 4, 70125, Bari, Italy}
\end{center}\vspace{1cm}

\begin{center} {\large ABSTRACT}
\end{center}
Motivated by the study of certain combinatorial properties of $(q,2)-$Fock space, we compute explicitly a sequence driven by the Catalan's convolution and parameterized by $1+q$. As an application of this explicit form, we calculate the number of pair partitions involved in the determination of the vacuum--moments of the field operator defined on the $(q,2)-$Fock space.\vspace{0.8cm}

\vspace{0.2cm}

\vspace{1.5cm}

\section{Introduction}\label{CCW-sec1}


In \cite{Catalan1887}, Catalan introduced what is nowadays known as the {\bf Catalan's convolution formula}, a fundamental result in combinatorics. This formula, as discussed by various authors (refer to  \cite{BowReg2014}, \cite{LarFre2003}, \cite{Regev2012}, \cite{Shapiro76}, and reference therein), provides a relationship between certain combinatorial sequences: for any $n\in\mathbb{N}^*$ and $m\in\{1,\ldots,n\}$,
\begin{align}\label{CCW01a}
&\sum_{i_1,\ldots,i_m\ge0; \,i_1+\ldots+i_m=n-m} C_{i_1}\cdot\ldots\cdot C_{i_m}\notag\\
\overset{j_k:=i_k+1,\,\forall k}{=} &\sum_{j_1,\ldots,j_m\ge1; \,j_1+\ldots+j_m=n} C_{j_1-1}\cdot\ldots\cdot C_{j_m-1}
=\frac{m}{2n-m}\binom{2n-m}{n}
\end{align}
Hereinafter, for any $k\in{\mathbb N}^*$,
$C_k$ denotes the $k-$th Catalan's number. Moreover,
this formula has found application in resolving various problems, as shown in \cite{Andrews1987}, \cite{Andrews2011}, \cite{Chen-Chu2009} and references provided therein.

The Catalan convolution formula is closely tied to the decomposition of a non--crossing pair partition into its {\it closed components} (defined in Definition \ref{CCW02} below). This relationship yields an interesting consequence: when one sums for all $m\in\{1,\ldots,n\}$ of the expression in \eqref{CCW01a}, the result is none other than the $n-$Catalan's number $C_n$:
\begin{align}\label{CCW01c}
&\sum_{m=1}^n\sum_{i_1,\ldots,i_m\ge0; \,i_1+\ldots+i_m=n-m} C_{i_1}\cdot\ldots\cdot C_{i_m}=\sum_{m=1}^n\frac{m}{2n-m}\binom{2n-m}{n}
=C_n
\end{align}
In this paper, we delve into investigation of the sum for all $m\in\{1,\ldots,n\}$ of the expression in \eqref{CCW01a} but with non trivial weights. More precisely,  we compute the following expression: \begin{align}\label{CCW01d}
&w_n(q):=\sum_{m=1}^n(1+q)^m\sum_{i_1,\ldots,i_m\ge0; \,i_1+\ldots+i_m=n-m} C_{i_1}\cdot\ldots\cdot C_{i_m},\qquad\forall n\in\mathbb{N}^*
\end{align}
Hereinafter, $q\in[-1,1]$.

The above expression naturally arises in the calculation of the vacuum expectation of a specific sequence of product of the creation and annihilation operators, defined on a particular Fock space, known as the $(q,2)-$Fock space. This Fock space holds significance for the following reasons:

From an analytical perspective, the $*-$algebra generated by the creation and annihilation operators within this Fock space generalizes the standard $q-$algebra: the scalar $q$ in the well--known $q-$communication relation is replaced by $q{\bf p}$ and ${\bf p}$ is a, usually non--trivial, projector.

From a combinatorial viewpoint, within the $(q,2)-$Fock space, the vacuum moments of the {\it field operators} (namely, the sums of the creation and annihilation operators) are determined by a family of sets ${\mathcal P}_n$'s. For each $n\ge3$, ${\mathcal P}_n$ is actually smaller (respectively, bigger) than $PP(2n):=$the totality of all pair partitions of the set $\{1,\ldots,2n\}$ (respectively, $NCPP(2n):=$ the totality of all non--crossing pair partitions of the set $\{1,\ldots,2n\}$).

In Section \ref{CCW-sec2}, we begin by introducing pair partitions and non--crossing pair partitions over any  totally ordered set, which severs as generalization of previously mentioned $PP(2n)$ and $NCPP(2n)$. We then define the {\it decomposition of non--crossing pair partitions} by using some well--known concepts, such as the {\it depth} of a pair within a given non--crossing pair partition, and the process of {\it gluing} pair partitions.

Section \ref{CCW-sec3} is devoted to presenting and demonstrating our main results in this paper. The Weighted Catalan convolution is deduced in terms of $(q,2)-$Fock space in Theorem \ref{CCW21}. Additionally, Theorem \ref{CCW22} provides an explicit representation of the Weighted Catalan convolution.

In Section \ref{CCW-sec4}, our main goal is to calculate $\vert{\mathcal P}_n\vert:=$the cardinalities of the sets ${\mathcal P}_n$, as an application of two theorems presented in Section \ref{CCW-sec3}. The explicit formulation to $\vert{\mathcal P}_n\vert$ is provided in Theorem \ref{CCW23}.

\section{Pair partitions and some related concepts; Elementary properties of $(q,2)-$Fock space} \label{CCW-sec2}

In this section, we provide technical background information to prepare for the statement and proof of our main results, namely, Theorems \ref{CCW21}, \ref{CCW22} and \ref{CCW23}.

\subsection{Pair partitions and some related concepts} \label{CCW-sec2-1}

Firstly, let's recall that for any $n\in{\mathbb N}^*$, a {\bf pair partition} of the set $\{1,\ldots,2n\}$ is defined as a collection of pairs $\{(i_h, j_h)\}_ {h=1}^n$, where,
\begin{align*}\{i_h,j_h:h=1,\ldots,n\}=\{1,\ldots,2n\};
\quad\ 1=i_1<\ldots<i_n<2; \quad\ i_h<j_h,\ \forall h\in\{1,\ldots,n\}
\end{align*}
Furthermore, a pair partition $\{(i_h,j_h)\}_ {h=1}^n$ is referred to as {\bf non--crossing} if the following equivalence holds for any $1\le k<h\le n$:
\begin{align*}i_k<i_h<j_k\ \iff\ i_k<j_h<j_k
\end{align*}
In this context, $i_h$'s are commonly known as the {\bf left indices}, and $j_h$'s as the {\bf right indices} of the pair partition $\{(i_h, j_h)\}_ {h=1}^n$.

For any $m,n\in{\mathbb N}^*$, it is common to denote:
\begin{align*}
&\{-1,1\}^m:=\big\{\text{functions on }\{1,\ldots,m\} \text{ and valued in }\{-1,1\} \big\}\notag\\
&\{-1,1\}^{2n}_+:=\big\{\varepsilon\in \{-1,1\}^{2n} :\sum_{h=1}^{2n}\varepsilon(h) =0,\, \sum_{h=p}^{2n}\varepsilon(h)\ge 0,\ \forall p\in \{1,\ldots,2n\}\big\}  \notag\\
&\{-1,1\}^{2n}_{+,*}:=\big\{\varepsilon\in \{-1,1\}^{2n}_+ :\sum_{h=p}^{2n}\varepsilon(h) =0 \text{ only for }p=1\big\}  \notag\\
&\{-1,1\}^{2n}_-:=\{-1,1\}^{2n}\setminus \{-1,1\}^{2n}_+
\end{align*}
Moreover, for any $\varepsilon\in\{-1,1\}^{2n}$, the following notations are frequently employed:
\begin{align*}
&PP(2n,\varepsilon):=\big\{\{(i_h, j_h)\}_ {h=1}^n\in PP(2n):\,\varepsilon^{-1}(\{-1\}) =\{i_h:h\in\{1,\ldots,n\}\}\big\}\notag\\
&NCPP(2n,\varepsilon):=NCPP(2n)\cap
PP(2n,\varepsilon)
\end{align*}
As shown in \cite{Ac-Lu96} and \cite{Ac-Lu2022a}, $PP(2n,\varepsilon)=\emptyset$ when $\varepsilon\in\{-1,1\}^{2n}_-$; however, for the case where $\varepsilon\in\{-1,1\}^{2n}_+$, the following observations hold:

$\bullet$ the set $NCPP(2n,\varepsilon)$  comprises exactly one element, which will be denoted as $\{(l^\varepsilon_h, r^\varepsilon_h)\}_ {h=1}^n$ throughout this paper;

$\bullet$ $\varepsilon\in\{-1,1\}^{2n}_{+,*}$ if and only if $r^\varepsilon_1=2n$;

$\bullet$ $PP(2n,\varepsilon)$ contains $\prod_{h=1}^n (2h-l^\varepsilon_h)$ elements;

$\bullet$ any $\{(i_h, j_h)\}_ {h=1}^n\in PP(2n,\varepsilon)$ satisfies the following:
\[i_h=l^\varepsilon_h,\ \forall h\in\{1,\ldots,n\};\quad \big\{j_h:h\in\{1,\ldots,n\} \big\}=\big\{r^\varepsilon_h:h\in\{1,\ldots,n\}\big\}\]

The concept of pair partitions and non--crossing pair partitions can be easily extended to any arbitrary totally ordered set. We will use the notation $B^V$ to denote the set of maps from $V$ to $B$ for any sets $B$ and $V$. Specifically, when $B=\{-1,1\}$ and $V=\{1,\ldots,m\}$ with $m\in\mathbb{N}^*$, $B^V$ coincides with $\{-1,1\}^m$ as mentioned earlier.

Moreover, for any $n\in\mathbb{N}^*$ and a set  $V:=\{ v_1,\ldots,v_{2n}\}$ ordered as $v_1<\ldots<v_{2n}$,
we introduce $\{-1,1\}^V_\pm$ and $\{-1,1\}^V_{+,*}$ as generalizations of $\{-1,1\}^{2n}_\pm$ and $\{-1,1\}^{2n}_{+,*}$ respectively:
\begin{align}\label{CCW19a3}
&\{-1,1\}^V_+:=\big\{\varepsilon\in \{-1,1\}^V :\sum_{h=1}^{2n}\varepsilon(v_h) =0,\, \sum_{h=p}^{2n}\varepsilon(v_h)\ge 0,\ \forall p\in \{1,\ldots,2n\}\big\}  \notag\\
&\{-1,1\}^V_-:=\{-1,1\}^V\setminus \{-1,1\}^V_+ \notag\\
&\{-1,1\}^V_{+,*}:=\big\{\varepsilon\in \{-1,1\}^V_+ : \sum_{h=p}^{2n}\varepsilon(v_h)= 0,\ \text{only for } p=1\big\}
\end{align}
Starting from this point, for any set $V=\{v_1,\ldots, v_{2n}\}$ consisting $2n$ elements, we will use the convention that $v_1<\ldots< v_{2n}$ unless otherwise specified.

\begin{remark} In the definition of $\{-1,1\}^V_+$, the condition $\sum_{h=p}^{2n}\varepsilon(v_h)\ge 0$ for any $p\in \{1,\ldots,2n\}$ can be equivalently replaced by $\sum_{h=1}^{k}\varepsilon(v_h)\le 0$ for any $k\in \{1,\ldots,2n\}$ since $\sum_{h=1}^{2n} \varepsilon(v_h)=0$.
\end{remark}

The concept of pair partition can be extended as follows: Given a set $V:=\{v_1,\ldots,v_{2n}\}$,
$\{(v_{i_h}, v_{j_h})\}_ {h=1}^n$ is a pair partition (or non--crossing pair partition) of $V$ if $\{({i_h}, {j_h})\}_ {h=1}^n$ is a pair partition (or non--crossing pair partition) of $\{1,\ldots,2n\}$. In other words, a pair partition, particularly a non--crossing pair partition, of the set $V:=\{v_1,\ldots,v_{2n}\}$ is defined based on the indices of $v_j$'s.

Generalizing the concepts of $PP(2n)$ and $NCPP(2n)$, we denote:
\begin{align*}
PP(V):=&\text{the totality of pair partitions of }V \notag\\
NCPP(V):=&\text{the totality of non--crossing pair partitions of }V
\end{align*}
As a result, $PP(2n)=PP(V)\Big\vert_{V=\{1,\ldots, 2n\}},\ NCPP(2n)=NCPP(V)\Big\vert_{V=\{1,\ldots,2n\}}$.
Additionally, we introduce
\begin{align}\label{CCW19a2}NCPP_*(V):=\big\{\{(v_{i_h}, v_{j_h})\}_ {h=1}^n\in NCPP(V):j_1=2n \big\}
\end{align}

Let $\tau: PP(V)\longmapsto \{-1,1\}^{V}$ be the following map: for any $\theta:=\{(v_{l_h},v_{r_h})\}_{h=1}^n\in PP(V)$, $\varepsilon:=\tau(\theta)$ is defined as
\begin{align*}
\varepsilon(v_j):=\tau(\theta)(v_j):=\begin{cases}
1,&\text{ if }j\in \{r_1,r_2,\ldots, r_n\}\\
-1,&\text{ if }j\in \{l_1,l_2,\ldots, l_n\}\\
\end{cases},\quad \forall j\in\{1,\ldots,n\}
\end{align*}
Clearly, $\tau$ maps $PP(V)$ into $\{-1,1\}^{V}_+$ due to the following observations:

$\bullet$ $\sum_{h=1}^{2n}\varepsilon(v_h)= \big\vert\big\{r_h:h \in\{1,\ldots,n\}\big\}\big\vert -\big\vert\big\{l_h:h\in\{1, \ldots,n\}\big\}\big\vert =n-n=0$;

$\bullet$ for any $p\in \{1,2,\ldots,2n\}$, it is clear that $\sum_{h=p}^{2n}\varepsilon(v_h)\ge 0$ because $r_j>l_j$ for any $j\in\{1,\ldots,n\}$.\\
On the other hand, for any $n\in\mathbb{N}^*$, $V:=\{v_1,\ldots,v_{2n}\}$ and $\varepsilon\in \{-1,1\}^{V}_+$, we can extend the definitions of the sets $PP(2n, \varepsilon)$ and $NCPP(2n,\varepsilon)$ as follows:
\begin{align*}
&PP(V,\varepsilon):=\tau^{-1}(\varepsilon)\notag\\
:=&\big\{\{(v_{l_h},v_{r_h})\}_{h=1}^n\in PP(V): \varepsilon(v_{l_h})=-1,\, \varepsilon(v_{r_h})=1, \forall h=1,\ldots,n\big\}\notag\\
&NCPP(V,\varepsilon):=\tau^{-1}(\varepsilon)\cap NCPP(V)
\end{align*}
Moreover, it is worth noting that the restriction of $\tau$ to the set $NCPP(V)$ establishes a bijection between $NCPP(V)$ and $\{-1,1\}^{V}_{+}$ as demonstrated in \cite{Ac-Lu96} that
$\big\vert\tau^{-1}(\varepsilon)\cap NCPP(V)\big\vert =1$ for any $\varepsilon\in \{-1,1\}^{V}_{+}$. Moreover, in a more restrictive sense, $\tau$ reduces a bijection between $NCPP_*(V)$ and $\{-1,1\}^{V}_{+,*}$.

In general, for any two elements $\{(v_{l_h},v_{r_h})\} _{h=1}^n$ and $\{(v_{i_h},v_{j_h})\}_{h=1}^n$ in $PP(V)$, the equality $\tau(\{(v_{l_h},v_{r_h})\} _{h=1} ^n) =\tau(\{(v_{i_h},v_{j_h})\}_{h=1}^n)$ holds if and only if $l_h=i_h$ for any $h\in\{1,\ldots,n\}$. Moreover, for any $\{(v_{l_h},v_{r_h})\}_{h=1}^n\in PP(V)$, as shown in \cite{Ac-Lu2022a}, there are $\prod_{h=1}^n (2h-l_h)$ elements of $PP(V)$ with the $\tau-$image as $\tau(\{(v_{l_h},v_{r_h})\}_{h=1}^n)$.

In the subsequent discussions, for any $n\in\mathbb{N}^*$ and the set $V:=\{v_1,\ldots,v_{2n} \}$, the following terminology will be used:

$\bullet$ $\tau(\theta)\in \{-1,1\}^{V}_+$ is referred to as the {\bf counterpart} of $\theta$ for any $\theta\in NCPP(V)$;

$\bullet$ the unique element of  $\tau^{-1} (\varepsilon)\cap NCPP(V)$ is called the {\bf counterpart} of $\varepsilon$ for any $\varepsilon\in \{-1,1\}^{V}_+$.

The counterpart of $\varepsilon\in\{-1,1\}^{V}_+$ mentioned above will be denoted as $\{(v_{l^\varepsilon _h}, v_{r^\varepsilon_h})\}_{h=1}^n$ which is identical to $\{(l^\varepsilon_h, r^\varepsilon_h)\}_ {h=1}^n$ when $v_k=k$ for all $k\in\{1,\ldots,2n\}$.

\begin{remark} Indeed, for any $n\in\mathbb{N}^*$, $V:=\{v_1,\ldots,v_{2n}\}$ order as $v_1<\ldots <v_{2n}$, and for any $\varepsilon\in \{-1,1\}^{V}_+$, all elements in $PP(V,\varepsilon)$ must share the same left--indices $\varepsilon^{-1}(\{-1\})$ and the same set of right--indices $\varepsilon^{-1}(\{1\})$.
\end{remark}

To state and prove our main results, we first introduce a {\it decomposition} of a non--crossing pair partition. To do this, we'll begin by revisiting the concepts of the {\it depth} of a pair in a given non--crossing pair partition and the {\it gluing} of pair partitions.

\begin{definition}\label{de-depth} For any $n\in\mathbb {N}^*$ and $V:=\{v_1,\ldots,v_{2n}\}$ ordered as $v_1<\ldots<v_{2n}$, for any $\theta:=\{(v_{l_h}, v_{r_h})\} _{h=1}^n \in NCPP(V)$ and $k\in \{1,\dots ,n\}$, we define
\begin{align*}
d_\theta(v_{l_k},v_{r_k}):=\left\vert \left\{h\in \{1,\dots ,n\} \ : \ l_{h}< l_{k}<r_{k}< r_{h}\right\}\right\vert
\end{align*}
as the {\bf depth} of the pair $(v_{l_k},v_{r_k})$  in $\theta$. Moreover, for any $\varepsilon\in\{-1,1\}^V _+$ with the counterpart of $\theta$, we denote $d_{\varepsilon}:=d_{\theta}$.
\end{definition}

It's important to note that we have defined {\it the depth of a pair} in a given pair partition $\theta$ only when $\theta$ is non--crossing.

Let

$\bullet$ $m\ge 2$ and $\{n_1,\ldots,n_m\} \subset\mathbb{N}^*$;

$\bullet$ $V_p:=\{v_{p,1},\ldots,v_{p,2n_p}\}$ be a set
with the order $v_{p,1}<\ldots<v_{p,2n_p}$, where $p\in\{1,\ldots,m\}$; moreover, $V_p$'s are assumed to be pairwise disjoint;

$\bullet$ ``$<$'' be a total order on $V:=\bigcup_p V_p$ such that its restriction to each $V_p$ follows the order $v_{p,1}<\ldots<v_{p,2n_p}$.

It's essential to remark that, under the order ``$<$'' on $V$, there is no specific rule for combining $v'\in V_p$ and $v''\in V_q$ when $p\ne q$.

Let's denote $n:=n_1+\ldots+n_m$. The disjointness of $V_p$'s implies that the cardinality of $V$ is equal to the sum of the cardinalities of $V_p$ for $p$ ranging from 1 to $m$, i.e., $\vert V\vert=\sum_{p=1}^m\vert V_p\vert=2\big(n_1+\ldots+n_m\big)=2n$. For any pair partitions $\theta^{(p)}:=\{(v_{p,l^{(p)}_h} ,v_{p, r^{(p)}_h})\}_{h=1}^{n_p}\in PP(V_p)$ with $p\in\{1, \ldots,m\}$, one can {\it glue} $\theta^{(p)}$'s together by introducing:
\begin{align*}
&\theta^{(1)}\uplus\ldots\uplus\theta^{(m)}:=\{(v_{1, l^{(1)}_h}, v_{1,r^{(1)}_h})\}_{h=1}^{n_1}\uplus \ldots \uplus \{(v_{m,l^{(m)}_h}, v_{m,r^{(m)}_h})\}_{h=1}^{n_m}\notag\\
:=&\big\{(v_{i_k},v_{j_k}):k\in\{1,\ldots,n\}
\text{ and each  }(v_{i_k},v_{j_k})\text{ is a certain }(v_{p,l^{(p)}_h}, v_{p,r^{(p)}_h})\big\}
\end{align*}
For example, if $V':=\{1,3,4,5\}$ and $V'':=\{2 ,6\}$, along with $\theta':=\{(1,5),(3,4)\}\in PP(V')$ and $\theta'':=\{(2,6)\}\in PP(V'')$, then $\theta' \uplus\theta''=\{(1,5),(3,4),(2,6)\}$. Moreover, for any non--empty subsets $P^{(1)}\subset PP(V_1),\ldots, P^{(m)}\subset PP(V_m)$, one defines
\begin{align}\label{CCW16g1}
&P^{(1)}\uplus\ldots\uplus P^{(m)}\notag\\
:=&\big\{\theta^{(1)}\uplus\ldots\uplus \theta^{(m)}:\theta^{(p)}\in P(V_p)\text{ for any } p\in\{1,\ldots,m\}\big\}
\end{align}

\begin{remark} 1) It is evident that $\theta^{(1)}\uplus\ldots\uplus \theta^{(m)}$ belongs to $PP(V)$ if, for any $p\in\{1, \ldots,m\}$, $\theta^{(p)}:=\{ (v_{p,l^{(p)}_h},v_{p, r^{(p)}_h})\}_{h=1}^{n_p}\in PP(V_p)$; moreover, ``$\uplus$'' is an associative operation.

2) It's worth noting that for any $P^{(1)}\subset PP(V_1)$ and $P^{(2)}\subset PP(V_2)$, the gluing $P^{(1)}\uplus P^{(2)}$ is contained in $PP(V_1\cup V_2)$; however,

$\bullet$ $P^{(1)}\uplus P^{(2)}$ may not necessarily be a subset of  $NCPP(V_1\cup V_2)$ even if $P^{(p)}\subset NCPP(V_p)$ for both $p\in\{1,2\}$; e.g., in the example mentioned earlier, both $\theta'$ and $\theta''$ belong to $NCPP(V')$ and $NCPP(V'')$ respectively, but the gluing $\theta' \uplus\theta''$ belongs to $PP(V'\cup V'')$ but not to  $NCPP(V'\cup V'')$.

$\bullet$ $PP(V_1)\uplus PP(V_2)$ is not necessarily equal to $PP(V_1\cup V_2)$; e.g., consider the case mentioned earlier in \eqref{CCW16g1}, $\{(1,2),(3,4), (5,6)\}$ belongs to $PP(V_1\cup V_2)$ but not to $PP(V_1)\uplus PP(V_2)$; in fact an element of $PP (V_1\cup V_2)$, says $\{(v_{i_h},v_{j_h})\}_{h=1} ^{n_1+n_2}$, belongs to $PP(V_1)\uplus PP(V_2)$ if and only if for any $h\in\{1,\ldots,n_1+n_2\}$, $v_{i_h}$ and $v_{j_h}$ belong to the {\bf same} $V_p$.
\end{remark}

In the following discussion, for simplicity, one uses notations like the {\it closed and open intervals} of $\mathbb{Z}$ as follows: for any $k,m\in\mathbb{Z}$ such that $k\le m$,
\begin{align*}
[k,m]:=\{h\in\mathbb{Z}:\,k\le h\le m\},\qquad ]k,m[:=\{h\in\mathbb{Z}:\,k< h< m\}
\end{align*}

For any $m\in\mathbb{N}^*$ and $\varepsilon\in \{-1,1\}^{m}$, let's define the function $S_\varepsilon:[1,m]\mapsto\mathbb{Z}$ as $S_\varepsilon(k):=\sum_{h=1}^{k}\varepsilon(h)$ for any $k\in[1,m]$. With this function, we can make the following observations:

$\bullet$ the set $S_\varepsilon^{-1}(\{0\})$ is contained in $\{2,4,\ldots\}$;

$\bullet$ in case $m=2n$, $\varepsilon\in \{-1,1\}^{2n}_+$ if and only if $S_\varepsilon\le0 \text{ and } 2n\in S_\varepsilon^{-1} (\{0\})$;

$\bullet$ in case $m=2n$, $\varepsilon\in \{-1,1\}^{2n}_{+,*}$ if and only if $S_\varepsilon\le0$ and $S_\varepsilon^{-1} (\{0\})$ is exactly $\{2n\}$.

For any $n\in\mathbb{N}^*$ and $\varepsilon\in\{-1,1\}^ {2n}_+$, we define $n_\varepsilon:=\vert S_\varepsilon^{-1} (\{0\})\vert $. The condition $S_\varepsilon(2n)=0$ ensures that $n_\varepsilon\ge1$. Moreover, since $S_\varepsilon^{-1}(\{0\})\subset\{2,4, \ldots,2n\}$, we can represent $S_\varepsilon^{-1}(\{0 \})$ as $\big\{2i_p:p\in[1,n_\varepsilon]\big\}$ with $1\le i_1<\ldots<i_{n_\varepsilon}=n$.
Furthermore, adapting the conventions $i_0:=0$,
It's easy to see that for any $p\in[0,n_\varepsilon-1]$,

$\bullet$ $\varepsilon_p:=$the restriction of $\varepsilon$ to the closed interval $V^\varepsilon _p:=[2i_p+1, 2i_{p+1}]$ must belong to $\{-1,1\}^{V^ \varepsilon_p}_{+,*}$ (see \eqref{CCW19a3} for the definition); in other words, $\varepsilon'_p\in\{-1,1\} ^{2(j_{p+1}-j_p)} _{+,*}$, where $\varepsilon'_p$ is obtained by shifting $\varepsilon_p$: $\varepsilon'_p(h):= \varepsilon(h+2i_p)$ for any $h\in[1,2(i_{p+1}-i_p)]$;

$\bullet$ $\varepsilon^0_p:=$the restriction of $\varepsilon$ to the open interval $V^{\varepsilon,0}_p :=]2i_p+1, 2i_{p+1}[$ must belong to $\{-1,1\}_+^{V^{ \varepsilon,0}_p}$, in other words, $(\varepsilon^0 _p)'$ is an element of $\{-1,1\} ^{2(i_{p+1}-i_p-1)}_+$ and is obtained by shifting $\varepsilon^0_p$: $(\varepsilon^0_p)'(h):= \varepsilon (h+2i_p+1)$ for any $h\in[1,2(i_{p+1}-i_p-1)]$.\\
Moreover, holds the following:
\begin{align}\label{CCW02a}
\{1,\ldots,2n\}=\bigcup_{p=1}^{n_\varepsilon}V^{\varepsilon}_{p-1}\ \text{ with }\ V^{\varepsilon}_{p-1}:=[2i_{p-1}+1, 2i_p],\quad \forall p\in [1,n_\varepsilon]
\end{align}

In the context of pair partition, for any $n\in{\mathbb N}^*$ and $\theta:=\{(l_h,r_h)\}_{h=1}^{n}\in NCPP (2n)$, there is a unique $n_\theta\in\{1,\ldots,n\}$ and $0=:j_0< j_1<\ldots<j_{n_\theta}= n$ such that
\begin{align*}
\{s\in\{1,\ldots,n\}:\,d_\theta(l_{s},r_{s})=0\}= \{j_p+1:\,p\in[0,n_\theta-1] \}
\end{align*}
This condition can also be expressed as:
\begin{align*}
l_{j_p+1}=2j_{p}+1\,\text{ and }\, r_{j_p+1}=2j_{p+1}, \qquad \forall p\in[0,n_\theta-1]
\end{align*}
Therefore, we have the following decomposition:
\begin{align}\label{CCW02b}
\theta:=\{(l_h,r_h)\} _{h=1}^{n}=\biguplus_{p=0} ^{n_\theta-1} \{(l_h,r_h)\}_{h=j_p+1} ^{j_{p+1}}
\end{align}

\begin{remark}\label{CCW03} It's evident that for any $n\in\mathbb{N}^*$ and $\theta:=\{(l_h,r_h)\}_{h=1}^{n} \in NCPP(2n)$, if \eqref{CCW02b} holds, then for any $p\in[0,n_\theta-1]$:

$\bullet$ $\{(l_h,r_h)\}_{h=j_p+1}^{j_{p+1}}$ runs over $NCPP_*([2j_p+1,2j_{p+1}])$ (see \eqref{CCW19a2} for the definition);

$\bullet$ $\{(l_h,r_h)\}_{h=j_p+2}^{j_{p+1}}$ runs over the set $NCPP(]2j_p+1,2j_{p+1}[)$ as  $\{(l_h,r_h)\}_{h=1}^{n}$ running over $NCPP(2n)$;

$\bullet$ for any $j_p+1<h\le j_{p+1}$, the pair $(l_h, r_h)$ locates inside of the pair $(l_{j_p+1}, r_{j_p+1} )$ and therefore $d_\theta(l_h,r_h)\ge1$.
\end{remark}

It's clear that, for any $n\in\mathbb{N}^*$ and $\varepsilon\in \{-1,1\}^{2n}_+$, denoting $\theta$ as the counterpart of $\varepsilon$, the $n_\varepsilon$, $i_p$'s in \eqref{CCW02a}, and $n_\theta$, $j_p$'s in \eqref{CCW02b} adhere to the following relationships:
\begin{align*}
n_\varepsilon=n_{\theta}\text{ and } i_p=j_p,
\quad\forall\,p\in[1,n_\theta]
\end{align*}

\begin{definition}\label{CCW02}For any $n\in\mathbb{N}^*$ and a $\theta:=\{(l_h,r_h)\} _{h=1}^{n}\in NCPP(2n)$ that satisfying \eqref{CCW02b},

$\bullet$ each $\{(l_h,r_h)\}_{h=j_p+1} ^{j_{p+1}}$ is called a {\bf closed components} of $\theta$;

$\bullet$ each $\{(l_h,r_h)\}_{h=j_p+2} ^{j_{p+1}}$ is termed as a {\bf open components} of $\theta$;

$\bullet$ \eqref{CCW02b} is named the {\bf closed components decomposition} of $\theta$.
\end{definition}

It's obvious that for any $n\in\mathbb{N}$ and $\theta\in NCPP(2n)$, the number of closed components in
$\theta$ is equal to the number of open components in $\theta$.

\subsection{$(q,2)-$Fock space}\label{CCW-sec2-2}

Now we begin to explore the $(q,2)-$Fock space, a specific instance of the $(q,m)-$Fock space introduced in \cite{YGLu2022a}. Let $\mathcal{H}$ be a Hilbert space with a scalar product $\langle \cdot,\cdot \rangle$ of the dimension greater than or equal to 2 (this convention will be maintained throughout), let $\mathcal{H}^{\otimes n}$ be its $n-$fold tensor product for any $n\ge 2$. One defines, for any $q\in[-1,1]$,

$\bullet$ $\lambda_1:={\bf 1}_{\mathcal{H}}$, i.e., the identity operator on $\mathcal{H}$;

$\bullet$ for any $n\in\mathbb{N}$, $\lambda_{n+2}$ be the {\bf linear} operator on $\mathcal{H}^{\otimes (n+2)}$ such that
\begin{equation*}
\lambda_{n+2}:={\bf 1}_{\mathcal{H}}^{\otimes n}\otimes \lambda_2\,\text{ and } \lambda_2(f\otimes g):=f\otimes g+q g\otimes f\,,\quad\forall f,g\in\mathcal{H}
\end{equation*}

The positivity of $\lambda_n$'s can be easily verified (see, for instance, \cite{BoKumSpe97}, \cite{Bo-Spe91}, \cite{Ji-Kim2006}) and which ensures that
\[
\mathcal{H}_n:=\text{ the completion of the } \big(\mathcal{H}^{\otimes n}, \langle \cdot,\lambda_{n} \cdot \rangle_{\otimes n}\big)/Ker\langle \cdot, \lambda_{n}\cdot \rangle_{\otimes n}
\]
is a Hilbert space, where $\langle \cdot,\cdot \rangle_{\otimes n}$ is the usual tensor scalar product. One denotes the scalar product of $\mathcal{H}_n$ as $\langle \cdot,\cdot \rangle_{n}$, where $\langle \cdot,\cdot \rangle_{1}:=\langle \cdot,\cdot \rangle $ and for any $n\ge2$,
\begin{equation*}
\langle F,G \rangle_{n}:=\langle F,\lambda_nG \rangle_{\otimes n} \,,\quad\forall F,G\in \mathcal{H}^{\otimes n}
\end{equation*}
or equivalently for any $n\in \mathbb{N}$, for any
$F\in \mathcal{H}^{\otimes (n+2)}$, $G\in \mathcal{H}^{\otimes n}$, and for any $f,g\in \mathcal{H}$,
\begin{align*}
\langle F,G\otimes f\otimes g \rangle_{n+2}:=&\langle F,G\otimes f\otimes g\rangle_{\otimes (n+2)}+ q\langle F,G\otimes g\otimes f\rangle_{\otimes (n+2)}
\end{align*}

\begin{definition}\label{(q,2)-Fock}Let $\mathcal{H}$ be a Hilbert space, and for any $q\in[-1,1]$ and $n\in\mathbb{N}^*$, let $\mathcal{H}_n$ be the Hilbert space as described above. The Hilbert space  $\Gamma_{q,2} (\mathcal{H}) :=\bigoplus_{n=0}^\infty\mathcal{H}_n $ is referred to as the {\bf $(q,2)-$Fock space} over $\mathcal{H}$, where $\mathcal{H}_{0} :=\mathbb{C}$. Moreover,

$\bullet$ the vector $\Phi:=1\oplus0\oplus0 \oplus\ldots$ is called the {\bf vacuum vector} of $\Gamma_{q,2}(\mathcal{H})$;

$\bullet$ for any $n\in\mathbb{N}^*$, $\mathcal{H}_n$ is named as the {\bf $n-$particle space}.
\end{definition}

Throughout, we will use simply the symbol $\langle \cdot,\cdot\rangle$ and $\Vert\cdot\Vert$ to denote the scalar product and the induced norm on both $\Gamma_{q,2}(\mathcal{H})$ and $\mathcal{H}_n$'s.

It's straightforward to observe the following {\bf consistency} of $\langle \cdot,\cdot \rangle_{n}$'s: for any non--zero $f\in\mathcal{H}$ and any $n\in\mathbb{N}^*$, holds the following:
\begin{equation*}
\Vert f\otimes F\Vert=0\text{ whenever } F\in\mathcal{H}_{n} \text{ verifying } \Vert F\Vert=0
\end{equation*}
This consistency guarantees that, for any $f\in\mathcal {H}$, the operator that maps $F\in\mathcal{H}_{n}$ to
$f\otimes F\in \mathcal{H}_{n+1}$ is a well--defined linear operator from $\mathcal{H}_{n}$ to $\mathcal{H}_{n+1}$.

\begin{definition}\label{creaOn(q,2)} For any $f\in \mathcal{H}$, the {\bf $(q,2)-$creation operator} (with the test function $f$) $A^+(f)$ is defined as such a {\bf linear} operator on $\Gamma_{q,2}(\mathcal{H})$ that
\begin{equation*}
A^+(f)\Phi:=f\,,\quad A^+(f)F:=f\otimes F,\quad \forall n\in\mathbb{N}^*\text{ and }F\in\mathcal{H}_n
\end{equation*}
\end{definition}

Let $\mathcal{H}$ be a Hilbert space and $q$ belongs to $[-1,1]$. The following elementary properties of $(q,2)-$Fock space (even more general, $(q,m)-$Fock space) over ${\mathcal H}$ are proved in \cite{YGLu2022a} and \cite{YGLu2022e}:

1) For any $f\in\mathcal{H}$, the $(q,2)-$creation operator $A^+(f)$ is bounded, and its norm is given by:
\begin{align*}
\Vert A^+(f)\Vert =\Vert f\Vert\cdot\begin{cases} \sqrt{1+q},&\text{ if }q\in[0,1];\\ 1,&\text{ if }q\in[-1,0)  \end{cases}
\end{align*}
Consequently, $A(f):=\big(A^+(f)\big)^*$ is well--defined, called the {\bf $(q,2)-$annihilation operator} with the test function $f\in \mathcal{H}$.

2) The $(q,2)-$annihilation operator $A(f)$ for any $f\in \mathcal{H}$ can equivalently be defined by its linearity with respect to the test function and the following equalities: $A(f)\Phi=0$ and
\begin{align}\label{CCW05g1}
&A(f)(g_1\otimes\ldots\otimes g_n)=A(f)A^+(g_1) \ldots A^+(g_1)\Phi\notag\\
=&\begin{cases}\langle f,g_1\rangle\Phi,&\text{ if } n=1;\\  \langle f, g_1\rangle g_2+q\langle f, g_2\rangle g_1,&\text{ if }n=2;\\  \langle f, g_1\rangle g_2\otimes\ldots\otimes g_n,&\text{ if }n> 2  \end{cases}
\end{align}
for any $n\in\mathbb{N}^*$ and $\{g_1,\ldots, g_n\}\subset\mathcal{H}$.

3) For any $f\in\mathcal{H}$, one has
$\Vert A(f)\big\vert_{\mathcal{H}_{1}}\Vert =\Vert A^+(f)\big\vert_{\mathcal{H}_{0}}\Vert =\Vert f\Vert$ and for any $n\in\mathbb{N}^*$,
\begin{align*}
&\Vert A(f)\big\vert_{\mathcal{H}_{n+1}}\Vert =\Vert A^+(f)\big\vert_{\mathcal{H}_n}\Vert\notag\\ &
\Vert A(f)A^+(f)\big\vert_{\mathcal{H}_n}\Vert  =\Vert A^+(f)A(f)\big\vert_{\mathcal{H}_n}\Vert =\Vert A(f)\big\vert_{\mathcal{H}_n}\Vert ^2
\end{align*}
hereinafter, for any $m\in{\mathbb N}$ and for any linear operator $B$ on the $(q,2)-$Fock space, $B\big\vert_{{\mathcal H}_m}$ is defined as its restriction to ${\mathcal H}_m$. Consequently
\begin{align*}
\Vert A(f)\Vert =\Vert A^+(f)\Vert\,;\quad \Vert A(f)A^+(f)\Vert  =\Vert A^+(f)A(f)\Vert =\Vert A(f)\Vert^2
\end{align*}

4) For any $m\in{\mathbb N}^*$ and $\varepsilon\in  \{-1,1\}^m$, the vacuum expectation
\begin{align*}
\big\langle \Phi,A^{\varepsilon(1)}(f_1)\ldots A^{\varepsilon(m)}(f_m)\Phi\big\rangle
\end{align*}
differs from zero only if $m$ is even, i.e. $m=2n$, and
$\varepsilon$ belongs to $\{-1,1\}^{2n}_+$;

5) For any $n\in{\mathbb N}^*$, $\varepsilon\in  \{-1,1\}^{2n}_+$ and $\{f_1,\ldots,f_{2n}\}\subset {\mathcal H}$, the vector $A^{\varepsilon(1)}(f_1)\ldots$ $ A^{\varepsilon(2n)}(f_{2n})\Phi$ belongs to ${\mathcal H}_0$ and, moreover, for any $C\in{\mathbb C}$, the following equivalence holds:
\begin{align}\label{CCW28b}
A^{\varepsilon(1)}(f_1)\ldots A^{\varepsilon(2n)}(f_{2n})\Phi=C\Phi\iff
\big\langle \Phi,A^{\varepsilon(1)}(f_1)\ldots A^{\varepsilon(2n)}(f_{2n})\Phi\big\rangle=C
\end{align}

\section{Weighted Catalan's convolution}\label{CCW-sec3}

Now, we are ready to state and prove the main results od this paper.

\begin{proposition}\label{CCW20h}
On the $(q,2)-$Fock space over $\mathcal{H}$,
\begin{align}\label{CCW20h0}
A(f)\big(A(f)A^+(f)\big)^nA^+(f)\Phi=(1+q)^n\Vert f\Vert ^{2(n+1)}\Phi,\quad \forall n\in\mathbb{N},\, f\in\mathcal{H}
\end{align}
or equivalently,
\begin{align}\label{CCW20h1}
\big\langle\Phi,  A(f)\big(A(f)A^+(f)\big)^nA^+(f)\Phi \big\rangle=(1+q)^n\Vert f\Vert ^{2(n+1)},\quad \forall n\in\mathbb{N},\, f\in\mathcal{H}
\end{align}
\end{proposition}

\begin{proof} \eqref{CCW28b} gives the equivalence between \eqref{CCW20h0} and \eqref{CCW20h1}. We see \eqref{CCW20h0} by applying the induction's argument.

For $n=0,1$, the equality in \eqref{CCW20h0} is a trivial consequence of \eqref{CCW05g1}. Suppose its validity for $n$, let's show that for any $f\in\mathcal{H}$,
\begin{align*}
A(f)\big(A(f)A^+(f)\big)^{n+1}A^+(f) \Phi= (1+q)^{n+1}\Vert f\Vert ^{2(n+2)}
\end{align*}
Indeed, since $A(f)A^+(f)A^+(f) \Phi\overset{\eqref{CCW05g1}}= (1+q)\Vert f\Vert ^2A^+(f) \Phi$, we can derive that
\begin{align*}
A(f)\big(A(f)A^+(f)\big)^{n+1}A^+(f) \Phi &=A(f)\big(A(f) A^+(f)\big)^{n}A(f)A^+(f)A^+(f) \Phi\notag\\
&=(1+q)\Vert f\Vert ^2 A(f)\big(A(f)A^+(f)\big)^{n} A^+(f)\Phi\notag\\
&=(1+q)\Vert f\Vert ^2 (1+q)^{n}\Vert f\Vert^{2(n+1)}
\end{align*}
where, the last equality is obtained thanks to the inductive assumption. \end{proof}

\begin{theorem}\label{CCW21} For any $n\in\mathbb{N}$ and $f\in\mathcal{H}$ with the unit norm, and for any $\varepsilon\in\{-1,1\}^{2n}_+$ whose counterpart has the following {\it closed components decomposition}:
\begin{align}\label{CCW21a}
\{(l_h^\varepsilon,r_h^\varepsilon)\}_{h=1}^n=\biguplus_{p=0}^{n_\varepsilon-1}\{(l_h^\varepsilon,r_h^\varepsilon)\} _{h=j_p+1}^{j_{p+1}}
\end{align}
it must be true that:
\begin{align}\label{CCW21b0}
\big\langle\Phi, A(f)A^{\varepsilon(1)}(f)\ldots A^{\varepsilon(2n)}(f) A^+(f) \Phi\big\rangle
=(1+q)^{n_\varepsilon}
\end{align}
Moreover, holds the following formula governed by the  weighted Catalan convolution:
\begin{align}\label{CCW21b}
&\sum_{\varepsilon\in\{-1,1\}^{2n}_+}\big\langle\Phi, A(f)A^{\varepsilon(1)}(f)\ldots A^{\varepsilon(2n)}(f) A^+(f) \Phi\big\rangle\notag\\
=&\sum_{m=1}^{n}(1+q)^{m}\sum_{j_1,\ldots,j_{m}
\ge1\atop j_1+\ldots+j_{m}=n} C_{j_1-1}\ldots C_{j_m-1}
\end{align}
\end{theorem}

\begin{proof} Let $V:=\{0,1,\ldots,2n,2n+1\}$ and $\bar\varepsilon:=(-1,\varepsilon(1),\ldots,\varepsilon (2n),1)$, then

$\bullet$ $\bar\varepsilon\in \{-1,1\}^V_{+,*}$ and its counterpart is $\{(l_h^\varepsilon,r_h^\varepsilon)\} _{h=1}^n \uplus \{(0,2n+1)\}$;

$\bullet$ $d_{\bar\varepsilon}(0,2n+1)=0$ and $d_{\bar\varepsilon}(l_h^\varepsilon,r_h^\varepsilon)= d_{\varepsilon}(l_h^\varepsilon,r_h^\varepsilon)+1\ge1$ for any $h\in\{1,\ldots,n\}$; moreover, \eqref{CCW21a} guarantees that
\begin{align*}
\{h:d_{\bar\varepsilon}(l_h^\varepsilon,r_h^\varepsilon)=1\}=\{h:d_{\varepsilon}(l_h^\varepsilon,r_h^\varepsilon)=0\}=\{j_p+1:p=0,1,\ldots,n_\varepsilon-1\}
\end{align*}
So, by defining $V^{\bar\varepsilon}_{0,1}:=\{0,2n+1, l_{j_p+1}^\varepsilon, r_{j_p+1}^\varepsilon: \, p\in[0,1, \ldots,n_\varepsilon-1]\}$ and $\bar\varepsilon_{0,1} :=$the restriction of $\bar\varepsilon$ on $V^{\bar\varepsilon}_{0,1}$,  $\bar\varepsilon_{0,1}$ must take the form $(-1,-1,1,-1,1,\ldots,-1,1,1)$. Thus, \eqref{CCW21b0} holds thanks to Proposition \ref{CCW20h}. Moreover, we have that:
\begin{align}\label{CCW21f}
\{-1,1\}^{2n}_+&=\bigcup_{m=1}^n\big\{\varepsilon\in \{-1,1\}^{2n}_+:\, \{(l_h^\varepsilon, r_h^\varepsilon) \}_{h=1}^n\text{ has } m\text{ close components}\big\}\notag\\
&=\bigcup_{m=1}^n\big\{\varepsilon\in \{-1,1\}^{2n}_+:\,
\{(l_h^\varepsilon, r_h^\varepsilon) \}_{h=1}^n\text{ has }
m\text{ open components}\big\}
\end{align}
and as mentioned in Remark \ref{CCW03},
\begin{align}\label{CCW21g}
&\big\vert\big\{\varepsilon\in \{-1,1\}^{2n}_+:\,
\{(l_h^\varepsilon, r_h^\varepsilon) \}_{h=1}^n\text{ has }
m\text{ open components}\big\}\big\vert\notag\\
=&\sum_{j_1,\ldots,j_{m}\ge1\atop j_1+\ldots+j_{m}=n}
C_{j_1-1}\ldots C_{j_m-1}
\end{align}
Therefore, we obtain \eqref{CCW21b} as an application of the formulae \eqref{CCW21b0}, \eqref{CCW21f}, and \eqref{CCW21g}. \end{proof}

As a straightforward application of \eqref{CCW21b} and properties of closed components decomposition of non--crossing pair partition, we find that, for any $n\in{\mathbb N}^*$,
\begin{align}\label{CCW21z1}
C_n=&\sum_{\varepsilon\in\{-1,1\}^{2n}_+}
\big\langle\Phi, A(f)A^{\varepsilon(1)}(f)\ldots A^{\varepsilon(2n)}(f) A^+(f) \Phi\big\rangle\notag\\
=&\sum_{m=1}^{n}\sum_{j_1,\ldots,j_m\ge1\atop j_1+\ldots+j_{m}=n} C_{j_1-1}\ldots C_{j_m-1}= \sum_{m=1}^{n}\sum_{i_1,\ldots,i_m\ge0\atop i_1+\ldots+i_{m}=n-m} C_{i_1}\ldots C_{i_m}
\end{align}
where, $A^+(f)$ and $A(f)$ are the creation and annihilation operators defined on the {\bf free} Fock space over ${\mathcal H}$ with unit form test function $f\in {\mathcal H}$.

Indeed, \eqref{CCW21b} implies that when $q=0$, we recover the second equality in \eqref{CCW21z1}. The first equality in \eqref{CCW21z1} holds because, for $q=0$, our Fock space is the {\it free Fock space}.

The following result provides an analogue of \eqref{CCW21z1} for general case of $q\in[-1,1] \setminus\{0\}$, in particular, for $q=1$.

\begin{theorem}\label{CCW22}For any $q\in[-1,1]\setminus \{0\}$ and $0\ne f\in{\mathcal H}$, let $w_0(q):=\Vert f\Vert$ and
\begin{align*}
w_n(q):=\sum_{\varepsilon\in\{-1,1\}^{2n}_+}
\big\langle\Phi, A(f)A^{\varepsilon(1)}(f)\ldots A^{\varepsilon(2n)}(f) A^+(f) \Phi\big\rangle,\quad \forall n\in \mathbb{N}^*
\end{align*}
then
\begin{align}\label{CCW22b}
w_n(q)&=\Vert f\Vert^{2n+2}\sum_{m=1}^{n}\left( 1+q\right) ^m \sum_{i_1,\ldots,i_m\geq0\atop i_1+\ldots+i_m=n-m} C_{i_{1}}\ldots C_{i_{m}}\notag\\
&=\Vert f\Vert^{2n+2}{1+q\over q}\Big({(1+q)^{2n-1}\over q^{n-1}}- \sum_{m=1}^nC_{m-1}{(1+q)^{2(n-m)}\over q^{n-m}}\Big),\qquad \forall n\in \mathbb{N}^*
\end{align}
In particular, for $q=1$, $w_0(1)=\Vert f\Vert^2$ and
\begin{align}\label{CCW22b0}
w_n(1)=\Vert f\Vert^{2n+2}\sum_{m=1}^{n}2^m \sum_{i_1,\ldots,i_m\geq0\atop i_1+\ldots+i_m=n-m} C_{i_1}\ldots C_{i_m} =\Vert f\Vert^{2n+2}\cdot \binom{2n}{n},\quad \forall n\in \mathbb{N}^*
\end{align}
i.e., the $2n-$th moment of the Arc-sine distribution with unit variance when $\Vert f\Vert=1$.
\end{theorem}

As an application of this result, one can find in \cite{YGLu2022b} the {\bf vacuum distribution} of the field operator $A(f)+A^+(f)$.

\begin{proof} Without loss of generality, we assume that $\Vert f\Vert=1$. In this case, the first equality in \eqref{CCW22b} is nothing but \eqref{CCW21b}, and we only need to consider the second.

For simplicity, $w_n(q)$ will be denoted as $w_n$ for all $n\in \mathbb{N}$.

In the case of $n=1$, we obtain, by recalling that $C_0:=1$, that:
\[{1+q\over q}\Big({(1+q)^{2n-1}\over q^{n-1}}- \sum_{m=1}^nC_{m-1}{(1+q)^{2(n-m)}\over q^{n-m}}\Big) \Big\vert_{n=1}={1+q\over q}\big( (1+q)-C_{0}\big)=1+q
\]

Let $W$ be the {\it generating function} of the sequence
$\{w_n\}_{n=0}^\infty$:
\begin{align}\label{CCW22c}
W(x):=\sum_{n=0}^{\infty}w_nx^n
=1+\sum_{n=1}^{\infty}w_nx^n
\end{align}
Since
\begin{align*}
0\le& w_n=\sum_{m=1}^{n}\left( 1+q\right) ^m \sum_{i_1,\ldots,i_m\geq0\atop i_1+\ldots+i_m=n-m} C_{i_{1}}\ldots C_{i_{m}}\notag\\
\le& 2^n\sum_{m=1}^{n} \sum_{i_1,\ldots,i_m\geq0\atop i_1+\ldots+i_m=n-m} C_{i_{1}}\ldots C_{i_{m}}
\overset{\eqref{CCW21z1}}=2^nC_n
\end{align*}
and since the series $\sum_{n=0}^\infty 2^nC_nx^n$ has a positive convergence radius, we conclude that the series $W$ given in \eqref{CCW22c} also has a positive convergence radius.

Moreover, for any $m\in\mathbb{N}^*$,
\begin{align}\label{CCW22d}
&\sum_{n=m}^\infty x^{n-m}\sum_{i_1,\ldots,i_m\ge0,\atop i_1+\ldots+i_m=n-m}C_{i_{1}}\ldots C_{i_{m}}=\sum_{k=0}^\infty x^k\sum_{i_1,\ldots,i_m\ge0,\atop i_1+\ldots+i_m=k}C_{i_{1}}\ldots C_{i_{m}}\notag\\
=&\Big(\sum_{k=0}^\infty C_{k}x^{k}\Big)^m=\Big(  \frac{1-\sqrt{1-4x}}{2x}\Big) ^m
\end{align}
and consequently, for any $x$ with a sufficiently small absolute value, we have:
\begin{align}\label{CCW22e}
W(x)&=1+\sum_{n=1}^\infty w_nx^n=1+\sum_{n=1}^\infty
x^n\sum_{m=1}^n(  1+q)^m\sum_{i_1,\ldots,i_m\ge0,\atop i_{1}+\ldots+i_m=n-m}C_{i_1}\ldots C_{i_m}\notag\\
& =1+\sum_{m=1}^\infty(1+q)^mx^m\sum_{n=m}^\infty
x^{n-m}\sum_{i_1,\ldots,i_m\ge0,\atop i_1+\ldots+i_m=n-m}C_{i_{1}}\ldots C_{i_{m}}\notag\\
& \overset{\eqref{CCW22d}}=1+\sum_{m=1}^\infty(1+q)^m x^m\Big(  \frac{1-\sqrt{1-4x}}{2x}\Big)  ^m\notag\\
& =\frac{1}{1-\frac{1+q}{2}\big(1-\sqrt{1-4x}\big) }= \frac{2}{1-q+(  1+q)  \sqrt{1-4x}}\notag\\
& =\frac{2\left(  1-q-(  1+q)\sqrt{1-4x}\right)  } {(1-q)^2-(1+q)^2(1-4x)}=\frac{-1}{2q}\cdot \frac{1-q-   (1+q)\sqrt{1-4x}}{1-\frac{(1+q)^2x}{q}}\notag\\
& =\frac{(1+q)\sqrt{1-4x}+q-1}{2q} \sum_{n=0}^\infty \left(\frac{(  1+q)^2x}{q}\right)^n
\end{align}
Using the notion of {\bf half--factorial} $(2n-1) !!:=(2n-1)\cdot(2n-3)\cdot\ldots\cdot3\cdot1$, it is straightforward to see that:
\begin{align*}
{d\over dx}\sqrt{1-4x}&=-2(1-4x)^{-{1\over2}}\\ {d^{n+1}\over dx^{n+1}}\sqrt{1-4x}&=-2^{n+1}\cdot (2n-1)!!\cdot (1-4x)^{-{{2n+1}\over2}},\qquad \forall n\in{\mathbb N}^*
\end{align*}
and consequently,
\begin{align}\label{CCW22n}
&\sqrt{1-4x}=1+{d\sqrt{1-4x}\over dx}\Big\vert_{x=0}x
+\sum_{n=1}^\infty{d^{n+1}\sqrt{1-4x}\over dx^{n+1}}\Big\vert_{x=0}{x^{n+1}\over (n+1)!}\notag\\
=&1-2x\sum_{n=0}^\infty C_nx^n
=1-2\sum_{n=1}^\infty C_{n-1}x^n
\end{align}
where, the penultimate equality holds because
\[{2^n\cdot (2n-1)!!\over (n+1)!}={n!\cdot 2^n\cdot  (2n-1)!!\over n!(n+1)!}={(2n)!\over n!(n+1)!}=C_n
\]
\eqref{CCW22n} trivially gives
\begin{align*}
{( 1+q)  \sqrt{1-4x}+q-1\over 2q}=1-{1+q\over q}\sum_{n=1}^\infty C_{n-1}x^n
\end{align*}
and by applying this to \eqref{CCW22e}, we obtain
\begin{align}\label{CCW22g}
W(x)=1+\sum_{n=1}^\infty w_nx^n=
\Big(1-{1+q\over q}\sum_{n=1}^\infty C_{n-1}x^n\Big)
\Big(1+\sum_{n=1}^\infty \left(\frac{(  1+q)^2}{q}\right)^nx^n\Big)
\end{align}
So, \eqref{CCW22b} is proven because the product of two series on the right hand side of \eqref{CCW22g} equals to:
\begin{align*}
1+&\Big({(1+q)^2\over q}-{(1+q)\over q}\Big)x\\
+&{(1+q)\over q}\sum_{n=2}^\infty\Big({(1+q)^{2n-1} \over q^{n-1}}-C_{n-1}-\sum_{m=1}^{n-1}C_{m-1}{(1+q)^ {2(n-m)}\over q^{(n-m)}}\Big)x^n
\end{align*}
i.e.
\[
1+(1+q)x+{(1+q)\over q}\sum_{n=2}^\infty
\Big({(1+q)^{2n -1}\over q^{n-1}}-\sum_{m=1}^{n}C_{m-1} {(1+q)^{2(n-m)} \over q^{(n-m)}}\Big)x^n
\]

Now let's present the proof of the part ``in particular''. Thanks to the additional assumption $q=1$, \eqref{CCW22e} provides us with
\begin{align}\label{CCW22h}
1+\sum_{n=1}^\infty w_n(1)x^n=\frac{1}{1-\frac{1+q}{2} \big(1-\sqrt{1-4x}\big) }\Big\vert_{q=1}={1\over \sqrt{1-4x}}
\end{align}
This is nothing else than the moment--generating function of the Arc--sine distribution with unit variance. Therefore,
\begin{align*}
&w_n(1)={1\over n!}{d^{n}\sqrt{1-4x}\over dx^{n}}\Big\vert_{x=0}={2^n\cdot (2n-1)!!\over n!}=
\binom{2n}{n},\qquad\forall n\in\mathbb{N}^*
\end{align*}      \end{proof}

\section{An application}\label{CCW-sec4}

As previously discussed \cite{YGLu2022e}, for any $n\in{ \mathbb N}^*$ and $\varepsilon\in\{-1,1\}^{2n}_+$, there exists a unique ${\mathcal P}_n(\varepsilon)\subset PP(2n)$ such that on the $(q,2)-$Fock space,
\begin{align}\label{CCW24}
A^{\varepsilon(1)}(f_1)\ldots A^{\varepsilon(2n)}(f_{2n})\Phi=\sum_{ \theta:=\{(v_{l^\varepsilon_h}, v_{r^\varepsilon_h})\}_{h=1}^n \in{\mathcal P}_n (\varepsilon)}q^{c(\theta)} \prod_{h=1}^n\langle f_{l_h},f_{r_h}\rangle \Phi
\end{align}
in particular (in fact, equivalently)
\begin{align}\label{CCW24a}
\big\langle \Phi, A^{\varepsilon(1)}(f_1)\ldots A^{\varepsilon(2n)}(f_{2n})\big\rangle=\sum_{ \theta:=\{(v_{l^\varepsilon_h}, v_{r^\varepsilon_h})\}_{h=1}^n \in{\mathcal P}_n (\varepsilon)}q^{c(\theta)} \prod_{h=1}^n\langle f_{l_h},f_{r_h}\rangle
\end{align}
where, by adapting the convention $\sum_{h=1} ^{0}:=0$, for any $\theta:=\{(v_{l_h}, v_{r_h})\} _{h=1}^{n}\in PP(V)$,
\begin{align*}
c(\theta):=c(\{(v_{l_h},v_{r_h})\}_{h=1}^{n}):=\sum_{h=1} ^{n-1}\vert\{j: l_h<l_j<r_h<r_j \}\vert
\end{align*}
is called the {\bf restricted crossing number of $\theta$} (see \cite{tM-mS2008} and reference within).
Moreover, \cite{YGLu2022e} provides the explicit form of the set ${\mathcal P}_n (\varepsilon)$: for any $\varepsilon\in\{-1,1\}^{2n}_+$,
\begin{align}\label{CCW15a1}
\mathcal{P}_n(\varepsilon):=\{(v_{l^\varepsilon_h}, v_{r^\varepsilon_h})\}_{h\le n:\, d_\varepsilon (l^\varepsilon_h, r^\varepsilon_h) \ge2}\uplus
PP(V_{0,1}^\varepsilon, \varepsilon_{0,1})
\end{align}
with $V^\varepsilon_{0,1}:=\big\{l^\varepsilon_h, r^\varepsilon_h: d_\varepsilon(l^\varepsilon_h, r^\varepsilon_h)\in\{0,1\}\big\}$ and $\varepsilon_{0,1}:=$the restriction of $\varepsilon$ to
$V^\varepsilon_{0,1}$. That is, an arbitrary element of $\mathcal{P}_n(\varepsilon)$ consists of (independently of the specific value of $q\in[-1,1]\setminus \{0\}$):

$\bullet$ all pairs $(v_{l^\varepsilon_h}, v_{r^ \varepsilon_h})$ with the depth greater or equal to 2;

$\bullet$ an arbitrary element of $PP(V_{0,1}^\varepsilon, \varepsilon_{0,1})$.\\
Additionally, $\mathcal{P}_n(\varepsilon)\cap \mathcal{P}_n(\varepsilon')=\emptyset$ whenever $\varepsilon$ and $\varepsilon'$ are two distinct elements of $\{-1,1\}^{2n}_+$.

A natural question arises: similar to the well--known results: $\big\vert NCPP(2n)\big\vert= C_n$ and $\big\vert PP(2n)\big\vert= (2n-1)!!$, what is the cardinality of the set $\mathcal{P}_n:=\cup_{\varepsilon \in\{-1,1\}^{2n}_+} \mathcal{P}_n(\varepsilon)$?

As an application of Theorems \ref{CCW21} and \ref{CCW22}, we can answer the above question as follows:

\begin{theorem}\label{CCW23}Let $u_0:=1$ and let, for any $n\in\mathbb{N}^*$, $u_n:=\big \vert {\mathcal P}_n\big\vert$ be the cardinality of the set $\mathcal{P}_n$.
Then $\{u_n\}_{n\in \mathbb{N}}$ verifies the system
\begin{align}\label{CCW23a}
u_{n+1}=\sum_{m=0}^{n} \binom{2m}{m}u_m,\qquad \forall n\in \mathbb{N}
\end{align}
with the initial conditions $u_0=1$. Moreover, $u_n$'s take the form:
\begin{align*}
u_0=u_1=1,\qquad u_n=f_n+g_n+ \sum_{m=1}^{n-1} f_{n-m}g_{m},\quad \forall n\ge 2
\end{align*}
where, for any $n\in\mathbb{N}^*$,
\begin{align}\label{CCW23b1}
f_n&:=\sum_{1\le m\le n+1; m\text{ is odd} } \binom{n+1}{m}2^{n+1-m}\cdot5^{m-1\over2} ,\qquad
g_n:=\begin{cases}
-3,\text{ if }n=1\\ -2C_{n-2},\text{ if }n\ge2
\end{cases}
\end{align}
\end{theorem}
\begin{proof} As mentioned previously, the construction of the sets $\mathcal{P}_n(\varepsilon)$'s is independent of the specific value of $q\in[-1,1] \setminus \{0\}$, so we can set $q=1$. This allows us to focus on the $(1,2)-$Fock space over a given Hilbert space $\mathcal{H}$.

In this case, for any unit norm vector $f\in\mathcal{H}$, \eqref{CCW24} gives us
\begin{align}\label{CCW24c}
A^{\varepsilon(1)}(f)\ldots A^{\varepsilon(2n)}(f)\Phi =\sum_{ \theta:=\{(v_{l^\varepsilon_h}, v_{r^\varepsilon_h})\}_{h=1}^n \in{\mathcal P}_n (\varepsilon)}\Phi= \big\vert{\mathcal P}_n (\varepsilon) \big\vert\Phi
\end{align}
In particular, with the help of the pairwise disjointness of ${\mathcal P}_n (\varepsilon)$'s, we have:
\begin{align}\label{CCW23c}
u_n=\big\vert{\mathcal P}_n\big\vert=\sum_{\varepsilon \in\{-1,1\}^{2n}_+}\big\vert{\mathcal P}_n (\varepsilon) \big\vert &=\sum_{\varepsilon\in\{-1,1\}^{2n}_+} \big\langle\Phi, A^{\varepsilon(1)} (f)\ldots A^{\varepsilon(2n)}(f)\Phi \big\rangle\notag\\
&= \big\langle\Phi, \big(A(f)+A^+(f)\big)^{2n}\Phi \big\rangle
\end{align}

Let's introduce a generalization of $\{-1,1\}^{2n}_{+,*}$ as follows: for any $n\in\mathbb{N}^*$ and $m\in\{1,\ldots,n\}$,
\begin{align*}
\{-1,1\}^{2n}_{+,m}:=\{\varepsilon\in \{-1,1\}^{2n}_+: \min S_\varepsilon^{-1}(\{0\})=2m\}
\end{align*}
Clearly, $\{-1,1\}^{2n}_{+,*}$ is just $\{-1,1\}^{2n}_{+,n}$; for any $m\in\{1,\ldots,n\}$
and $\varepsilon\in \{-1,1\}^{2n}_{+,m}$, it must be true that $S_\varepsilon(2m)=0$, and $S_\varepsilon(p) \ne0$ for any $p<m$. As a result,
\begin{align}\label{CCW23d}
\{-1,1\}^{2n}_{+}=\bigcup_{m=1}^n\{-1,1\}^{2n}_{+,m} \ \text{ and } \  \{-1,1\}^{2n}_{+,m}\cap\{-1,1\}^{2n} _{+,m'} =\emptyset \text { if }m=m'
\end{align}

For any $n\in\mathbb{N}^*$ and $m\in\{1,\ldots,n\}$, for any $\varepsilon\in \{-1,1\}^{2n}_{+,m}$, we introduce the following:
\begin{align*}
&\varepsilon_1:=\text{the restriction of  }\varepsilon \text{ on the set }\{1,\ldots,2m\}\notag\\
&\varepsilon_2(h):=\varepsilon(2m+h),\qquad\forall h\in\{1,\ldots,2n-2m\}
\end{align*}
It follows that $\varepsilon_1\in \{-1,1\}^{2m}_{+,*}$ and $\varepsilon_2\in \{-1,1\}^{2(n-m)}_{+}$; moreover,
as $\varepsilon$ running over $\{-1,1\}^{2n}_{+}$, $\varepsilon_1$ (and $\varepsilon_2$) runs over $\{-1,1\}^{2m}_{+,*}$ (and $\{-1,1\}^{2(n-m)}_+$) respectively.

For any $\{f_1,\ldots,f_{2n}\} \subset\mathcal{H}$, we can utilize the fact that $\sum_{h=2m+1}^{2n} \varepsilon(h)=0$ to obtain that:
\[A^{\varepsilon(2m+1)}(f_{2m+1})\ldots A^{\varepsilon(2n)}(f_{2n})\Phi=\langle\Phi, A^{\varepsilon(2m+1)}(f_{2m+1})\ldots A^{\varepsilon(2n)}(f_{2n})\Phi\rangle\Phi    \]
This implies that
\begin{align}\label{CCW23d3}
&\sum_{\varepsilon\in\{-1,1\}^{2n}_+}\big\langle\Phi, A^{\varepsilon(1)} (f_{1})\ldots A^{\varepsilon(2n)}(f_{2n})\Phi \big\rangle\notag\\
\overset{\eqref{CCW23d} }=&\sum_{m=1}^n \sum_{\varepsilon\in\{-1,1\}^{2n}_{+,m}} \big\langle\Phi, A^{\varepsilon(1)} (f_{1})\ldots A^{\varepsilon(2n)}(f_{2n})\Phi \big\rangle\notag\\
=&\sum_{m=1}^n\sum_{\varepsilon_1\in\{-1,1\}^{2m}_{+,*}} \big\langle\Phi, A^{\varepsilon_1(1)} (f_{1})\ldots A^{\varepsilon_1(2m)}(f_{2m})\Phi \big\rangle\notag\\
&\hspace{0.5cm}\sum_{\varepsilon_2\in\{-1,1\}^{2(n-m)}_+} \big\langle\Phi, A^{\varepsilon_2(1)}(f_{2m+1}) \ldots A^{\varepsilon_2(2(n-m))}(f_{2n})\Phi \big\rangle
\end{align}
By setting all $f_k$ to be equal to $f\in{\mathcal H}$ with unit norm for all $k\in\{1,\ldots,2n\}$ in the right hand side of \eqref{CCW23d3}, we observe that,   for any $m\in\{1,\ldots,n\}$,

$\bullet$ the sum $\sum_{\varepsilon_1\in\{-1,1\} ^{2m}_{+,*}} \big\langle\Phi, A^{\varepsilon_1(1)} (f_{1})\ldots A^{\varepsilon_1(2m)}(f_{2m})\Phi \big\rangle$ is nothing else than $w_{m-1}$ defined in \eqref{CCW22b0} and which is equal to  $\binom{2m-2}{m-1}$;

$\bullet$ the sum $\sum_{\varepsilon_2\in\{-1,1\} ^{2(n-m)}_+} \big\langle\Phi, A^{\varepsilon_2(1)} (f_{2m+1}) \ldots A^{\varepsilon_2(2(n-m))}(f_{2n})\Phi \big\rangle $ is nothing else than $u_{n-m}$, as introduced in \eqref{CCW23c}.\\
By applying these observations to \eqref{CCW23d3}, we arrive at the result given in \eqref{CCW23a}.

Moreover, as a consequence, the moment--generating function of the sequence $\{u_n\} _{n=0}^\infty$ can be calculated as follows:
\begin{align*}
U(x):=&\sum_{n=0}^\infty u_nx^n=1+x\sum_{n=0}^\infty u_{n+1}x^n\overset{\eqref{CCW23a}}=1+x\sum_{n=0}^\infty x^n\sum_{k=0}^{n}\binom{2k}{k}u_{n-k}\notag\\ =&1+x\sum_{k=0}^\infty \binom{2k}{k}x^k\sum_{n=k}^\infty
u_{n-k}x^{n-k}=1+xU(x)\sum_{k=0}^\infty w_k x^k
\overset{\eqref{CCW22h}}=1+{xU(x)\over \sqrt{1-4x}}
\end{align*}
By resolving this equation, we obtain
\begin{align}\label{CCw24d}
U(x)=&{\sqrt{1-4x} \big( \sqrt{1-4x}+x\big)\over 1-4x-x^2}={1-4x +x\sqrt{1-4x}\over 2\sqrt{5}}\Big({1\over \sqrt{5}-2-x} +{1\over \sqrt{5}+2+x}\Big)\notag\\
=&{1-4x +x\sqrt{1-4x}\over 2\sqrt{5}}\sum_{n=0}^\infty x^n\Big( (2+\sqrt{5})^{n+1}-(2-\sqrt{5})^{n+1}\Big)
\end{align}
Moreover, thanks to the facts that $f_1+g_1\overset{\eqref{CCW23b1}}=1$,
\[1-4x +x\sqrt{1-4x}\overset{\eqref{CCW22n}}=1-3x-2\sum_{n=2}^\infty C_{n-2}x^n
\overset{\eqref{CCW23b1}}=1+\sum_{n=1}^\infty g_nx^n
\]
and for any $n\in{\mathbb N}^*$,
\[{1\over 2\sqrt{5}}\Big( (2+\sqrt{5})^{n+1} -(2-\sqrt{5})^{n+1}\Big)=\sum_{1\le m\le n+1; m\text{ is odd} }\binom{n+1}{m} 2^{n+1-m}\cdot 5^{m-1\over 2} \overset{\eqref{CCW23b1}}=f_n\]
we complete the proof by rewriting \eqref{CCw24d} as:
\begin{align*}
U(x)=&\big(1+\sum_{n=1}^\infty g_nx^n\big)
\big(1+ \sum_{n=1}^\infty f_n x^n\big)\notag\\
=&1+(f_1+g_1)x+\sum_{n=2}^\infty (f_n+g_n+\sum_{m=1} ^{n-1}f_{n-m}g_{m})x^n
\end{align*}       \end{proof}

The following table gives a comparison of the cardinalities of $NCPP(2n)$, $\mathcal{P}_n$ and $PP(2n)$ for $1\le n\le6$:
\begin{equation*}
\begin{tabular}{|l||*{3}{c|}}\hline
n&\makebox[9.5em]{$\vert NCPP(2n)\vert=C_n$} &\makebox[9.5em]{$\vert \mathcal{P}_n\vert$} &\makebox[9.5em]{$\vert PP(2n)\vert=(2n-1)!!$}\\\hline\hline
1 &\makebox[9.5em]{1}&\makebox[9.5em]{1} &\makebox[9.5em]{1}\\\hline
2 &\makebox[9.5em]{2}&\makebox[9.5em]{3}    &\makebox[9.5em]{3}\\\hline
3 &\makebox[9.5em]{5}&\makebox[9.5em]{11}    &\makebox[9.5em]{15}\\\hline
4 &\makebox[9.5em]{14}&\makebox[9.5em]{43}    &\makebox[9.5em]{105}\\\hline
5 &\makebox[9.5em]{42}&\makebox[9.5em]{173}    &\makebox[9.5em]{945}\\\hline
6 &\makebox[9.5em]{132}&\makebox[9.5em]{707}    &\makebox[9.5em]{10395}\\\hline
\end{tabular}
\end{equation*}

\end{document}